\documentclass[12pt]{amsart}
\usepackage{graphicx}
\usepackage{amsfonts}
\usepackage{epsf}
\usepackage{amssymb}
\usepackage{amsmath}
\usepackage{amscd}
\usepackage{amsthm}
\usepackage{tikz}
\usepackage{pdfpages}
\usepackage{fancyhdr}
\usepackage{setspace}
\usepackage{hyperref}
\usepackage[all]{xy}
\usetikzlibrary{matrix}
\usepackage{verbatim}
\usepackage{enumerate}
\usepackage{mathrsfs}
\usepackage{pinlabel}
\usepackage{lscape}
\usepackage{color}
\usepackage{mathtools}
\usepackage[margin=3cm,marginpar=2.5cm]{geometry}

\newtheorem{theorem}{Theorem}[section]
\newtheorem{proposition}[theorem]{Proposition}
\newtheorem{lemma}[theorem]{Lemma}
\newtheorem{corollary}[theorem]{Corollary}

\newtheorem{claim}[theorem]{Claim}

\theoremstyle{definition}
\newtheorem{definition}[theorem]{Definition}

\newtheorem{example}[theorem]{Example}

\renewcommand{\t}{\mathfrak{t}}

\newcommand{\calc}{\mathcal{C}}

\newcommand{\Z}{\mathbb{Z}}

\newcommand{\Q}{\mathbb{Q}}

\newcommand{\Spinc}{{\rm Spin}^c}

\newcommand{\spinc}{\text{spin}^c}

\newcommand{\defeq}{\vcentcolon=}
\newcommand{\lift}{\tilde\rho}

\newcommand{\fs}{\mathfrak{s}}
\newcommand{\ft}{\mathfrak{t}}

\def\co{\colon\thinspace}

\makeatletter
\newtheorem*{rep@theorem}{\rep@title}
\newcommand{\newreptheorem}[2]{%
\newenvironment{rep#1}[1]{%
 \def\rep@title{#2~\ref{##1}}%
 \begin{rep@theorem}}%
 {\end{rep@theorem}}}
\makeatother

\newreptheorem{theorem}{Theorem}
\newreptheorem{lemma}{Lemma}
\newreptheorem{question}{Question}
\newreptheorem{corollary}{Corollary}

\begin{document}
\makeatletter
\providecommand\@dotsep{5}
\makeatother
\rhead{\thepage}
\lhead{\author}
\thispagestyle{empty}

\raggedbottom
\pagenumbering{arabic}
\setcounter{section}{0}



\title{Linear independence in the rational homology cobordism group}
\author{Marco Golla}
\address{CNRS, Laboratoire de Math\'ematiques Jean Leray, Nantes, France}
\email{marco.golla@univ-nantes.fr}

\author{Kyle Larson}
\address{Alfr\'ed R\'enyi Institute of Mathematics, Budapest, Hungary}
\email{larson@renyi.mta.hu}

\begin{abstract}
We give simple homological conditions for a rational homology 3--sphere $Y$ to have infinite order in the rational homology cobordism group $\Theta^3_\Q$, and for a collection of rational homology spheres to be linearly independent.
These translate immediately to statements about knot concordance when $Y$ is the branched double cover of a knot, recovering some results of Livingston and Naik.
The statements depend only on the homology groups of the 3--manifolds, but are proven through an analysis of correction terms and their behavior under connected sums. 
\end{abstract}
\maketitle

\begin{section}{Introduction}
Given a ring $R$, the $R$--homology cobordism group $\Theta^3_R$ is defined as the set of $R$--homology $3$--spheres up to \emph{smooth} $R$--homology cobordism, with the operation of connected sum; that is, $\Theta^3_R$ the set of equivalence classes of 3--manifolds $Y$ such that $H_*(Y;R) = H_*(S^3;R)$, modulo \emph{smooth} cobordisms $W: Y\to Y'$ such that $H_*(W,Y;R) = H_*(W,Y';R) = 0$.
In this paper we will consider the two cases $R = \Z$ and $R = \Q$, namely the (integral) homology cobordism group and the rational homology cobordism group, respectively.
These groups have been extensively studied (e.g. \cite{CassonHarer, Froyshov, OSz, Lisca-sums}) and are still a remarkable source of open problems (e.g. Problem 4.5 in~\cite{Kirbyslist}); for instance, it is not yet known if there are $n$-torsion elements in $\Theta^3_\Q$ for $n\neq 2$, or if there are $\Q$--summands or quotients in $\Theta^3_\Q$.

The simplest obstruction to a rational homology 3--sphere $Y$ being trivial in the rational homology cobordism group $\Theta^3_\Q$ is the following: if $Y$ bounds a rational homology ball then the order of its first homology must be a square.
One of our goals is to provide similarly simple obstructions to a rational homology sphere having finite order in $\Theta^3_\Q$.
Note that the order of the first homology of $\#^{2m} Y$ is always a square, and so the above obstruction is of little help.
Nonetheless, there do exist obstructions depending only on the homology of $Y$.
More generally, we can give a statement about linear independence in $\Theta^3_\Q / \Theta^3_\Z$ for an infinite family $Y_1,Y_2,\dots$, which depends only on the homology of $Y_i$.
(By the quotient $\Theta^3_\Q / \Theta^3_\Z$ we really mean $\Theta^3_\Q / \psi(\Theta^3_\Z)$, where $\psi \co \Theta^3_\Z \rightarrow \Theta^3_\Q$ is the natural homomorphism: an integral homology sphere is also a rational homology sphere.)

\begin{theorem}\label{linear independence}
Let $\{Y_{i}\}_{i\ge 1}$ be a collection of $\Z/2\Z$--homology $3$--spheres for which there exists a collection $\{p_i\}_{i\ge 1}$ of pairwise distinct primes congruent to $3$ modulo $4$, and a collection $\{n_i\}_{i\ge 1}$ of odd, positive integers, such that:
\begin{itemize}
\item the $p_i$--primary part of $H_1(Y_{i})$ is $\Z/p_i^{n_i}\Z$, and
\item the $p_j$--primary part of $H_1(Y_i)$ is $0$ for each $i\neq j$.
\end{itemize}
Then $\{Y_{i}\}_{i\ge 1}$ is a linearly independent set in $\Theta^3_\Q / \Theta^3_\Z$.
In particular, $Y_i$ has infinite order in $\Theta^3_\Q / \Theta^3_\Z$ for each $i$.
\end{theorem}

For example, a simple special case is when $H_1(Y_{i}) \cong \Z/p_i\Z$ for each $i$.

While the hypotheses of the theorem only depend on the homology of $Y$, the proof depends on some non-vanishing results for correction terms (in Seiberg--Witten or Heegaard Floer theory).
In fact, we show that certain correction terms are non-vanishing modulo 2, and this will explain why we get obstructions to being trivial in $\Theta^3_\Q / \Theta^3_\Z$ and not just $\Theta^3_\Q$.
From Theorem~\ref{linear independence} it follows immediately that $\Theta^3_\Q/\Theta^3_\Z$ contains subgroups isomorphic to $\Z^\infty$, a fact first shown in \cite{KL} following work in \cite{HLR}.
Another proof appears in \cite{AL} using the description of the subgroup of $\Theta^3_\Q$ generated by lens spaces from \cite{Lisca-sums}.

Instead of focusing only on the homology of a rational homology sphere $Y$, we can add the geometric condition that $Y$ is obtained by integral surgery on a knot in $S^3$. Then we get the following slightly different obstruction to being of finite order in $\Theta^3_\Q$.

\begin{theorem}\label{finite order}
Let $Y$ be a rational homology $3$--sphere and $n$ an odd integer.
If $Y$ is $n$--surgery on a knot in $S^3$ and $|n| \not\equiv 1 \pmod{8}$, then $Y$ has infinite order in $\Theta^3_\Q$.
\end{theorem}

\subsection{Applications to knot concordance}

Since the double cover of $S^3$ branched over a slice knot bounds a rational homology ball, the above results translate into statements concerning the knot concordance group $\calc$.
For a knot $K$ in $S^3$, let $Y_K$ denote the double cover of $S^3$ branched over $K$.
From Theorem~\ref{linear independence} we immediately obtain the following corollary, recovering results of Livingston and Naik.

\begin{corollary}[\cite{LN1, LN2}]\label{concordance}
Let $\{K_{i}\}_{i\ge 1}$ be a collection of knots such that $\{Y_{K_i}\}_{i\ge 1}$ satisfies the conditions of Theorem~\ref{linear independence}.
Then $\{K_{i}\}_{i\ge 1}$ is a linear independent set in $\calc$, and in particular, $K_i$ has infinite order in $\calc$ for each $i$.
\end{corollary}


In fact we get linear independence in $\calc/\calc_1$, where we use $\calc_1$ to denote the subgroup of $\calc$ generated by knots with determinant 1.
In the same spirit, Theorem~\ref{finite order} gives a corollary about knot concordance.

\begin{corollary}\label{concordance finite order}
Let $K$ be a knot in $S^3$ such that $Y_K$ is $n$--surgery on a knot in $S^3$ and $|n| \not\equiv 1 \pmod{8}$. Then $Y$ has infinite order in $\calc$.
\end{corollary}

Note that this can be used to show certain knots (e.g.~the knot $5_1$) have infinite order in $\calc$ for which the Livingston--Naik obstruction does not apply. (That the knot $5_1$ has infinite order follows more simply from the fact that its signature is non-zero.)

Livingston and Naik proved their results by studying the Casson--Gordon invariants of a knot and the potential metabolizers on sums of $H_1(Y_K)$.
In their argument they use the Casson--Gordon invariant associated to a character $\chi \co H_1(Y_K) \to S^1$.
Since they make crucial use of the information coming from the knot in working with 
the $0$--surgery along the lift of $K$ in $Y_K$, their arguments do not directly give statements about $\Theta^3_\Q$.
However, our work will make essential use of their study of metabolizers (see Section~\ref{algebra}).

\subsection*{Notation}
Throughout all manifolds will be smooth and compact, and unless otherwise stated all homology groups will be calculated with integer coefficients.
We will also use the shorthand 
$mY \defeq \#^m Y = Y \# \dots \# Y$, where there are $m$ copies of $Y$ on the right.
Note that $m$ can be negative, where we interpret $-|m|Y = |m|(-Y)$, and $-Y$ is $Y$ with reversed orientation.
We will use a similar notation for connected sums of knots.

\subsection*{Organization of the paper}
In Section~\ref{s:metabolizers} we recap some properties of metabolizers and linking forms, and prove Proposition~\ref{metabolizers}, the algebraic core of the paper. In Section~\ref{s:proofs}, we prove Theorems~\ref{linear independence} and~\ref{finite order} as well as their corollaries, as a consequence of a slightly more general result, Theorem~\ref{main}.

\subsection*{Acknowledgements}
The authors would like to thank Charles Livingston and Daniel Ruberman for some helpful and encouraging correspondence. We would also like to thank the referee for their very useful comments.
\end{section}

\begin{section}{Correction terms and metabolizers}\label{s:metabolizers}

This section comprises a recap on linking forms, $\rho$--invariants, and correction terms, as well as a long algebraic detour.
The latter is the technical core of the paper.

\subsection{Linking forms, metabolizers, $\rho$--invariants, and correction terms}\label{recap}

The discussion here follows closely~\cite{OwensStrle-Definite}.
In what follows, given a rational homology 3--sphere $Y$, we will identify $H^2(Y)$ with $H_1(Y)$ using Poincar\'e duality, leaving the isomorphism implicit.

Recall that if $Y$ is a closed, oriented 3--manifold, then $\Spinc(Y)$ is an affine space over $H^2(Y)$.
The $\rho$--\emph{invariant} of $Y$ is the map $\rho_Y: \Spinc(Y) \to \Q/2\Z$ defined as:
\[
 \rho_Y(\ft) = \frac{c_1(\fs)^2 - \sigma(W)}4,
\]
where $(W,\fs)$ is any $\spinc$ 4--manifold whose boundary is $(Y,\ft)$.

The \emph{linking form} on $Y$ is a non-degenerate bilinear form $\lambda_Y: H^2(Y) \times H^2(Y) \to \Q/\Z$ defined as follows.
Let $W$ be a 4--manifold with $\partial W = Y$;
since $Y$ is a rational homology sphere, $H^2(Y)$ is torsion, and in particular $H^2(W,Y;\Q) \to H^2(W;\Q)$ is onto.
Given $x,y$ in the image of the map $H^2(W) \to H^2(Y)$, consider two lifts $\bar x, \bar y \in H^2(W,Y;\Q)$, and define
\[
\lambda_Y(x,y) = -\langle \bar x \cup \bar y, [W,Y] \rangle \in \Q/\Z.
\]

Since one can always find a 4-manifold $W$ such that $H^2(W) \to H^2(Y)$ is onto, we can define $\lambda_Y(x,y)$ for each $x,y \in H^2(Y)$.
It is well-known that both $\rho_Y$ and $\lambda_Y$ do not depend on the choice of $W$.
%

In particular, if $W$ is a rational homology ball, since the intersection form is trivial, one sees that $\lambda_Y(x,y) = 0$ for each pair $x,y$ in the image of $H^2(W)\to H^2(Y)$;
call $M_W$ the image of this map.
From the long exact sequence of the pair $(W,Y)$, using Poincar\'e--Lefschetz duality and the universal coefficient theorem, one also sees that the order of $H^2(Y)$ is the square of the order of $M_W$.
By definition, $M_W$ is a metabolizer for the linking form $\lambda_Y$.
It follows from the non-degeneracy of $\lambda_Y$ that $H^2(Y)/M_W$ is in fact isomorphic to $M_W$ (see \cite[Theorem 2.7]{LN2}).

For convenience, we will now work only with $\Z/2\Z$--homology spheres, usually named $Y$.
The advantage of working with a $\Z/2\Z$--homology sphere $Y$ is that the unique conjugation-invariant $\spinc$ structure, i.e., the unique spin structure on $Y$, serves as an origin for the affine identification of $\Spinc(Y)$ with $H^2(Y)$.
We will denote the $\spinc$ structure associated to the class $x\in H^2(Y)$ under this identification as $\ft_x$.
Note that $c_1(\ft_x) = 2x$.
(Since $2$ is invertible, the first Chern class induces an isomorphism.)

The linking pairing is always determined by the $\rho$--invariant via a polarization formula:
\[
\lambda_Y(x,y) = -\frac{\rho_Y(\ft_x + \ft_y) - \rho_Y(\ft_x) - \rho_Y(\ft_y) + \rho_Y(\ft_0)}{2} \in \Q/\Z.
\]
Conversely, one can see that $\lambda_Y$ determines $\rho_Y$.
Choose a simply connected, \emph{spin} 4--manifold $W$ bounding $Y$ (such a manifold does indeed exist; see~\cite{Kaplan}).
This implies that $H^2(W) \to H^2(Y)$ is onto.
Since $H^2(W)$ is torsion-free, $\spinc$ structures on $W$ are determined by their first Chern class; since $W$ is spin, $\ft_0\in\Spinc(Y)$ extends to the $\spinc$ structure $\fs_0$ with $c_1(\fs_0) = 0$.
Fix $x\in H^2(Y)$, and choose an element $x'$ that maps to $x$ in $H^2(W) \rightarrow H^2(Y)$.
Let $\fs_{x'}$ be an extension of $\ft_{x}$ to $W$, such that $c_1(\fs_{x'}) = 2x'$;
as above, we let $\bar x \in H^2(W,Y;\Q)$ be the preimage of $x'$, viewed as a rational cohomology class, under the canonical map $H^2(W,Y;\Q) \to H^2(W;\Q)$.
We can then compute:
\[
\rho_Y(\ft_x) - \rho_Y(\ft_0) = \frac{c_1(\fs_{x'})^2 - \sigma(W)}4 - \frac{c_1(\fs_0)^2 - \sigma(W)}4 = x'\cdot x' = \bar x\cdot \bar x  =  -\lambda_Y(x,x).
\]
(This is more algebraic in nature than what appears from this exposition; see~\cite{OwensStrle-Definite} for more details.)
Now observe that if we chose a different $x''$ in $H^2(W)$ mapping to $x$, with corresponding extension $\fs_{x''}$, then $\fs_{x''} = \fs_{x'} + h$, for $h \in H^2(W;\Z)$, rather than $h \in H^2(W;\Q)$;
the first Chern class of $\fs_{x''}$ is now $c_1(\fs_{x''}) = c_1(\fs_{x'}) + 2h = 2x'+2h$.
Since $W$ is spin, $h\cdot h$ is an even integer, and therefore the quantity
\[
c_1(\fs_{x''})^2- c_1(\fs_{x'})^2 = (2x'+ 2h)^2 - (2x')^2 = 8(x'\cdot h) + 4h\cdot h
\]
is divisible by 8.
It follows that the difference $\rho_Y(\ft_x) - \rho_Y(\ft_0)$ is indeed well-defined modulo 2, rather than modulo 1.

In fact, since every non-degenerate linking pairing $\lambda \co H\times H \to \Q/\Z$ admits an even presentation~\cite[Theorem 6]{Wall-linking}, the argument above shows that the associated quadratic form $\rho(x) \defeq \lambda(x,x)$ is well defined as a map $H \to \Q/2\Z$.

\begin{example}\label{n-surgery}
Let $K$ be a knot in $S^3$, and let $Y = S^3_n(K)$ be the 3--manifold obtained as $n$--surgery along $K$, for $n>0$ odd.
For instance, when $K$ is the unknot, $S^3_n(K) = L(n,n-1) = -L(n,1)$.
Then, using the surgery handlebody $X_n(K)$ obtained by attaching a single 2--handle to $B^4$ along $K$, with framing $n$, one easily
computes that $\rho_Y(\ft_0) = \frac{n-1}4$.

The formula also holds true when $n$ is even, provided we choose the spin structure $\ft_0$ to be the one that extends to a $\spinc$ structure $\fs$ on $X_n(K)$, such that $\langle c_1(\fs), A \rangle = n$, where $A$ is a generator of $H_2(X_n(K))$.
\end{example}

Translating in terms of the $\rho$--invariant the metabolizer condition, we see that if $W$ is a rational homology ball bounding $Y$, then $\rho_Y(\ft_x) = 0$ for each $x\in M_W$. 
(Here we are using the identification $\Spinc(Y) = H^2(Y)$ as above; $M_W$ corresponds to the set of $\spinc$ structures on $Y$ that extend to $W$.)

We now define what it means for $\lift$ to be a lift of $\rho$ that is invariant under rational homology $\spinc$ cobordism.
For the sake of brevity, we will just speak of $\lift$ as a lift of $\rho$.

\begin{definition}
We say that $\lift$ is a \emph{lift} of $\rho$ if it is a map from $\spinc$ 3--manifolds to $\Q$; it assigns to each $\spinc$ rational homology 3--sphere $(Y,\ft)$ a rational number $\lift(Y,\ft)$, such that:
\begin{itemize}
\item[(i)] $\lift(Y,\ft) \equiv \rho_Y(\ft) \pmod 2$,
\item[(ii)] $\lift(Y,\ft) = \lift(Y,\bar\ft)$, where $\bar\ft$ is the conjugate of $\ft$,
\item[(iii)] $\lift(Y\#Y',\ft\#\ft') = \lift(Y,\ft) + \lift(Y',\ft')$, and
\item[(iv)] if $W$ is a rational homology cobordism from $Y$ to $Y'$, and $\fs$ is a $\spinc$ structure on $W$, then $\lift(Y,\fs|_Y) = \lift(Y',\fs|_{Y'})$.
\end{itemize}
\end{definition}

\begin{proposition}\label{p:formal_d}
Any lift $\lift$ of the $\rho$--invariant as above satisfies the following properties:
\begin{enumerate}
\item $\lift(S^3,\ft_0) = 0$, where $\ft_0$ is the unique $\spinc$ structure on $S^3$;
\item if $W$ is a rational homology ball with boundary $Y$, and $\fs\in\Spinc(W)$, then $\lift(Y,\fs|_Y) = 0$.
\end{enumerate}
\end{proposition}

\begin{proof}
Both properties follow easily from the last two conditions imposed on $\lift$:
\begin{enumerate}
\item since $(S^3\#S^3, \ft_0\#\ft_0) = (S^3,\ft_0)$, (iii) above implies $2\lift(S^3,\ft_0) = \lift(S^3,\ft_0)$, and hence the latter vanishes;
\item remove a small open ball from $W$, thus obtaining a rational homology cobordism $W'$ from $S^3$ to $Y$, and let $\fs' = \fs|_{W'}$; then $\fs'|_Y = \fs|_Y$ and $\fs'|_{S^3} = \ft_0$; from (iv) above and the previous point, $\lift(Y,\fs|_Y) = \lift(S^3,\ft_0) = 0$.\qedhere
\end{enumerate}
\end{proof}

Such lifts do indeed exist.
Examples are provided by correction terms in Seiberg--Witten theory \cite{Froyshov} or, equivalently, in Heegaard Floer theory \cite{OSz}.
From now on, we choose to adopt the Heegaard Floer setup, and use the lift $\lift(Y,\ft) = d(Y,\ft)$;
however, we remark that we are only using properties (i)--(iv) above, rather than a specific lift.

We now give a lemma about extensions of spin structures.

\begin{lemma}[cf. {\cite[Proposition 4.2]{Stipsicz}} or {\cite[Theorem 6]{KL}}]\label{l:spin-extension}
If $Y$ is a $\Z/2\Z$--homology sphere, then the $\spinc$ structure $\t_0$ on $Y$ that is also spin extends to all rational homology ball fillings of $Y$.
\end{lemma}

\begin{proof}
Suppose $W$ is a rational homology ball filling of $Y$; then, as noted above, $H^2(Y)$ has order $m^2$ for some $m$, where the image $M_W$ of the restriction map $\Spinc(W)\to\Spinc(Y)$ has cardinality $m$.
(Again, we will silently identify $\Spinc(Y)$ with $H^2(Y)$.)

Moreover, the restriction map is conjugation-equivariant, and so $M_W$ is invariant under conjugation;
since $M_W$ has odd cardinality, there is a fixed point under conjugation; that is, $M_W$ contains a spin structure.
Since $Y$ is a $\Z/2\Z$--homology sphere, there is a unique spin structure $\t_0$, which therefore belongs to $M_W$.
This proves that $\t_0$ extends to all rational homology ball fillings of $Y$, as claimed.
\end{proof}

We summarize the previous discussion in the following catch-all statement,
where we have applied Poincar\'e duality so as to now state things in terms of homology.
\begin{theorem}\label{summaryGL}
Let $Y$ be a $\Z/2\Z$--homology sphere that bounds a rational homology ball, and let $G = H_1(Y)$.
Then the order of $G$ is a square, and there exists a metabolizer $M \subset G$  such that $d(Y,\ft_x)= 0$ for each $x\in M$.
That is, there is a subgroup $M\subset G$ such that $G/M \cong M$ and $d(Y,\ft_x) = 0$ for each $x\in M$;
this implies in turn that $\rho_Y(\ft_x) \equiv 0 \pmod{2}$ and $\lambda_Y(x,y) \equiv 0 \pmod 1$ for each $x,y\in M$.
\end{theorem}


We will be interested in applying Theorem~\ref{summaryGL} to connected sums ${2m}Y$ with $Y$ satisfying certain homological conditions.
In Section~\ref{algebra} we will show that the existence of such a metabolizer for ${2m}Y$ puts strong restrictions on the function $\bar{d}: G\to \Q$ defined by:
\[
 \bar{d}: g \mapsto d(Y,\ft_g) - d(Y,\ft_0).
\]
(The notation is borrowed from~\cite{HLR}.)
Then we will show that for the 3--manifolds under consideration these restrictions on $\bar{d}$ are not satisfied.

We conclude this recap with an algebraic lemma that applies to linking forms.

\begin{lemma}\label{linking}
Let $p$ be an odd prime and $n$ an odd integer, and $H = \Z/p^n\Z$;
let $\rho: H \to \Q/2\Z$ be a quadratic form associated to the bilinear, non-degenerate form $\lambda: H\times H \to \Q/\Z$.
Then $\rho$ identically vanishes on the subgroup of order $p^{(n-1)/2}$, but not on the subgroup of order $p^{(n+1)/2}$.
\end{lemma}

Indeed, it is well-known that under the assumptions of the lemma, there are only two non-degenerate pairings $\lambda$ on $\Z/p^n\Z$ up to (pairing-preserving) homomorphism (for instance, see~\cite{Wall-linking}).
In fact, when $p \equiv 3 \pmod 4$, the two non-degenerate pairings are $\lambda_\pm(x,y) = \pm xy/p^n \in \Q/\Z$.

\begin{proof}
From now on we identify the set of elements in $\Q/\Z$ killed by $p^n$ with $\Z/p^n\Z$.
Note that the image of $\lambda$ falls in this subgroup.

Since $H$ is cyclic, $\lambda$ is determined by the value $c:=\lambda(1,1)$; namely, we have $\lambda(x,y) = cxy$.
Since the pairing is non-degenerate, $c$ is invertible in $\Z/p^n\Z$.

Let $s = (n-1)/2$;
recall that the subgroup of $H$ of order $p^{s}$ consists of all elements of the form $p^{s+1}x$, and that the subgroup of order $p^{s+1}$ consists of elements of the form $p^sx$.
We have:
\begin{align*}
\rho(p^{s+1}x) &= \lambda(p^{s+1}x, p^{s+1}x) = p^{2s+2}\lambda(x,x) = p^n \cdot px^2c = 0,\\
\rho(p^s) &= \lambda(p^s,p^s) = p^{2s}\lambda(1,1) = p^{n-1}\cdot c \neq 0,
\end{align*}
since $c$ is invertible.
\end{proof}

\subsection{Algebraic detour}\label{algebra}

In this section we give a purely algebraic result adapted from an argument of Livingston and Naik \cite[Section 3]{LN2}.
We also borrow the algebraic formalism from the Appendix of \cite{HLR}, which contains a proof of the analogue of Proposition~\ref{metabolizers} for the case $n=2$ (their argument is essentially the inductive step in the proof of the proposition). 
While we are most interested in the behavior of correction terms under connected sums, we consider more generally arbitrary functions $f \co \Z/p^n\Z \rightarrow \Q$ that satisfy some of the properties of correction terms (when thought of as the function $\bar d \co H_1(Y) \rightarrow \Q$ defined above). 
Hence we will impose the assumptions that $f(0)=0$ and $f(-g) = f(g)$ for $g \in \Z/p^n\Z$.
We extend such an $f$ to a function $f^{(m')} \co (\Z/p^n\Z)^{\oplus m'} \rightarrow \Q$ by $f^{(m')}(g_1,g_2, \dots, g_{m'}) \defeq f(g_1)+f(g_2)+ \dots + f(g_{m'})$, for $g_i \in \Z/p^n\Z$.
We will restrict to the case when $p$ is a prime congruent to 3 modulo 4, $n$ is odd, and $m' = 2m$ is even. 
We will also assume that $G = (\Z/p^n\Z)^{\oplus 2m}$ is equipped with a non-degenerate pairing $\lambda$, compatible with $f^{(2m)}$, in the sense that $f^{(2m)}(g_1,\dots,g_{2m}) \equiv \sum_i \lambda(g_i,g_i) \pmod 2$.

\begin{proposition}\label{metabolizers}
Let $f \co \Z/p^n\Z \rightarrow \Q$ be a function as above.
If $f^{(2m)}$ vanishes on a metabolizer $M$ for $(G,\lambda)$, then $f$ is identically zero on the subgroup of $\Z/p^n\Z$ generated by $p^{(n-1)/2}$.
\end{proposition}

Note that if $n=1$ this means that $f$ is identically zero on the whole group $\Z/p\Z$.

The proof will be by induction on the order of the subgroup of $\Z/p^n\Z$.
The base case is the vanishing of $f$ on the trivial subgroup; then we will then show that if $f$ vanishes on a subgroup of order $p^r < p^{(n+1)/2}$, then it also vanishes on the subgroup of order $p^{r+1}$.

We begin with two preparatory lemmas, and with some notation.

\begin{lemma}\label{generating set}
After perhaps permuting the order of the summands of $(\Z/p^n\Z)^{\oplus 2m}$, $M$ is generated by $\ell$ elements of the form
\[
\omega_i = (\underbrace{0,\dots,0}_{i-1\textrm{ zeros}},p^{a_i},\omega_{i,i+1},\dots,\omega_{i,\ell},\omega_{i,\ell+1},\dots,\omega_{i,2m})
\]
for $1\leq i\leq \ell$, where $p^{a_i}$ divides $\omega_{i,j}$ for $j>i$.
Furthermore we can arrange that $\omega_{i,i}$ divides $\omega_{i+1,i+1}$, i.e., that the sequence $a_i$ is non-decreasing.\qedhere
\end{lemma}

\begin{proof}
This is a corollary of the Gauss--Jordan elimination algorithm (see the proof of Theorem A.2 in \cite{HLR}).
\end{proof}

%
%

We establish the following notation: let $M_{p^i} = \{ x \in M \mid p^ix=0\}$, and let $k_j$ denote the number of generators $\omega_i$ with leading term $p^j$.
Note that $\ell = \sum_{j=0}^{n-1}k_j$.
For convenience, let $k_n = k_0 = k$.

\begin{lemma}\label{k_j}
The sequence $(k_j)$ is symmetric, i.e. $k_j = k_{n-j}$ for $0 \le j \le n$, and moreover $2m = \ell + k$.
\end{lemma}

\begin{proof}
In this proof, indices in all direct sums, all products, and all sums run from $0$ to $n$.

Note that $M$ is isomorphic to $\bigoplus_{j} (\Z/p^{n-j}\Z)^{\oplus k_j}$, and that $G/M$ is isomorphic to $\bigoplus_{j} (\Z/p^{j}\Z)^{\oplus k_j}$.
Since $M$ is a metabolizer for the linking form, $G/M \cong M$, and by the classification of finite abelian groups we obtain the first property for each $0 < j < n$.
Since we defined $k_n = k_0$, this holds also for $j=0, n$.

From the previous paragraph, we see that the order of $M$ is $\prod_{j} p^{(n-j)k_j}$, and the order of $G/M$ is $\prod_{j} p^{jk_j}$.
Since the product of the orders of $M$ and $G/M$ is the order of $G$, which is $p^{2nm}$, taking logarithms (in base $p$) one gets:
\[
2nm = \sum_j (n-j)k_j + \sum_j jk_j = n\sum_j k_j = n(\ell + k_n) = n(\ell + k).\qedhere
\]
\end{proof}

For convenience, we label subgroups of $\Z/p^n\Z$ by their order: we let $H_j$ be the subgroup of order $p^j$, i.e. the subgroup generated by $p^{n-j}$.

\begin{proof}[Proof of Proposition~\ref{metabolizers}]
As announced, the proof will be by induction on the order of the subgroup $H_j$ in $\Z/p^n\Z$.
The basis of the induction is simply the vanishing of $f$ on the identity of $G$, which is obvious by the definition of $f$.

The inductive step is as follows: if $f$ is identically zero on the subgroup $H_{n-r}$ of $\Z/p^n\Z$ generated by $p^{r}$, for some $r \geq (n+1)/2$, then $f$ is identically zero on the subgroup $H_{n-r+1}$ (generated by $p^{r-1}$).

For convenience, let $\overline\ell = k_0 + \dots + k_{r-1}$, and $\overline k = 2m-\overline{\ell}$.

Consider the elements $\omega_i$ from Lemma~\ref{generating set} whose leading terms are less than or equal to $p^{r-1}$, i.e. those for which $a_i \le r-1$.
In the notation above, these are the elements $\omega_1, \dots, \omega_{\overline \ell}$.

Multiply $\omega_i$ by $p^{r-1-a_i}$, so that the leading term of $p^{r-1-a_i}\omega_i$ is $p^{r-1}$, and add them to obtain
\[
z = p^{r-1-a_1}\omega_1 + \dots + p^{r-1-a_{\overline\ell}}\omega_{\overline\ell} = (p^{r-1},\dots,p^{r-1},b_1,\dots,b_{\overline k}).
\]
The first $\overline\ell$ coordinates are equal to $p^{r-1}$, each $b_i$ is a multiple of $p^{r-1}$.

Since $k_j=k_{n-j}$ by Lemma~\ref{k_j}, and since $n-r \le r-1$ by assumption, we have
\[
\overline k =  k_0 + \sum_{j=r}^{n-1}k_j = \sum_{i=0}^{n-r} k_j \leq \sum_{i=0}^{r-1} k_j = \overline\ell.
\]
By construction we have $z \in M_{p^{n-r+1}}$ and $f^{(2m)}(z) = 0$.

Consider the subgroup $H_{n-r+1}$ of $\Z/p^n\Z$ generated by $p^{r-1}$, and the subgroup $H_{n-r}$ generated by $p^{r}$.
Now $H_{n-r+1}$ has order $p^{n-r+1}$, and $H_{n-r}$ is a subgroup of index $p$ in $H_{n-r+1}$.
By the inductive hypothesis, we know that $f$ vanishes on $H_{n-r}$, and we want to show that $f$ is zero on all of $H_{n-r+1}$.
Therefore, it is enough to prove that $f$ vanishes on $H_{n-r+1}\setminus H_{n-r}$.

We identify $H_{n-r+1}$ with $\Z/p^{n-r+1}\Z$; under this identification $H_{n-r}$ is the subgroup of $\Z/p^{n-r+1}\Z$ generated by $p$.
The group of multiplicative units $(\Z/p^{n-r+1}\Z)^*$ of $\Z/p^{n-r+1}\Z$ consists exactly of those elements not in the subgroup generated by $p$, i.e. $H_{n-r+1}\setminus H_{n-r}$;
it has order $p^{n-r+1}-p^{n-r}=p^{n-r}(p-1)=p^{n-r}(2q)$, where $p=2q+1$ as before.
For convenience, call $\overline{q} = p^{n-r}q$.

Therefore $(\Z/p^{n-r+1}\Z)^*/\{\pm 1\}$ is a cyclic group of order $\overline{q}$, an odd number.
Choose a generator $a$ for this group, so that
\[
(\Z/p^{n-r+1}\Z)^*/\{\pm 1\} = (H_{n-r+1}\setminus H_{n-r})/\{\pm1\} = \{1, a, \dots, a^{\overline{q}}\}.
\]
We define $f_i \defeq f(\pm a^ip^{r-1})$ for each $a^i$ in $H_{n-r+1}\setminus H_{n-r}$.
There are exactly $\overline{q}$ such $f_i$.
We consider a relation space
\[
R = \left\{(\alpha_0, \dots, \alpha_{\overline{q}-1}) \in \Q^{\overline{q}} : \sum^{\overline{q}-1}_{i=0}\alpha_i f_i=0\right\}.
\]
We will prove in Claim~\ref{R} below that $R$ is a $(\overline{q}-1)$--dimensional subspace unless all $f_i$ vanish.
We define a map $\psi \co M_{p^{n-r+1}} \rightarrow R$ by
\[
\psi(x_1,\dots,x_{2m})=(\alpha_0,\dots, \alpha_{\overline{q}-1}),
\]
where $\alpha_i$ is the number of coordinates $j$ such that $x_j = \pm a^ip^{r-1} \pmod{p^n}$.

Note that here the generator $a$ of $(\Z/p^{n-r+1}\Z)^*/\{\pm 1\}$ acts on $\Q^{\overline{q}}$ by shifting the indices by $1$, and acts on $M_{p^{n-r+1}}$ by multiplication by $a$; we call $\tau$ both these actions.
By Claim~\ref{equivariance} below, the map $\psi$ is $\tau$--equivariant.
Hence we have an identification of $\Q^{\overline{q}}$ with $\Q[t]/(t^{\overline{q}}-1)$ such that the linear subspace spanned by the image of $\psi$ is an ideal.
Furthermore, under this identification $\psi(z)$ is given by the polynomial $h_{z}(t)= \beta_0+\beta_1t+\dots + \beta_{\overline{q}-1}t^{\overline{q}-1}$,
where $\beta_0 \geq \overline{\ell}$ is the number of coordinates in $z$ equal to $p^{r-1}$, all the $\beta_i$ are non-negative, and
\[
\sum^{\overline{q}-1}_{i=1}\beta_i \leq \overline{k} \le \overline{\ell} \leq \beta_0.
\]

To complete the inductive step, we show that the ideal $(h_{z})$ is all of $\Q[t]/(t^{\overline{q}}-1)$, which is the case unless $h_{z}$ vanishes at a $\overline{q}$--th root of unity.
Let $\zeta$ be such a root. Considering the real part of $h_{z}$, and using the fact that $\beta_0$ is at least as large as the sum of the remaining coefficients, we see that this is only possible if $\beta_i\cdot \zeta^i=-\beta_i$ for all $1\leq i\leq \overline{q}-1$ and at least one $\beta_i$ is nonzero.
This implies that $\zeta$ is an even root of unity, but $\overline{q}$ is odd.
In conclusion, we must have that all of the $f_i$ are zero, and so $f$ vanishes on $H_{n-r+1}\setminus H_{n-r}$, as required.
\end{proof}

We conclude the section by proving the two claims made in the proof above.

\begin{claim}\label{equivariance}
The function $\psi$ is $\tau$--equivariant. That is, $\psi \circ \tau = \tau \circ \psi$.
\end{claim}

\begin{proof}
Suppose $\psi(x_1,\dots,x_{2m})=(\alpha_0,\dots, \alpha_{q-1})$ as above.
Then $\psi(ax_1,\dots,ax_{2m})=(\alpha_0',\dots, \alpha_{q-1}')$, where $\alpha_i'$ is the number of indices $j$ such that $ax_j = \pm a^i p^{n-1} \pmod{p^n}$. But this is precisely the number of indices such that $x_j = \pm a^{i-1} p^{n-1} \pmod{p^n}$, and so $\alpha_i'=\alpha_{i-1}$. Hence $\psi(ax_1,\dots,ax_{2m})=(\alpha_{q-1},\alpha_0,\dots,\alpha_{q-2})$ as required.
\end{proof}

\begin{claim}\label{R}
$R$ is a $(q-1)$--dimensional subspace unless all $f_i = 0$.
\end{claim}

\begin{proof}
This follows immediately from the observation that $R$ is the kernel of the linear map $\Q^q \rightarrow \Q$ defined by the $1 \times q$ matrix $[f_0\, f_1\, \dots \, f_{q-1}]$.
\end{proof}

\end{section}
\begin{section}{Proofs of Theorems~\ref{linear independence} and~\ref{finite order}, and their corollaries}\label{s:proofs}

We adopt the notation and conventions of Section~\ref{recap}; in particular, we label $\spinc$ structures on a $\Z/2\Z$--homology sphere $Y$ by classes in $H_1(Y)$; for instance, the spin structure on $Y$ will always be $\ft_0$.

As a warm-up, we start with the proof of Theorem~\ref{finite order}.
Recall that we need to prove that $S^3_n(K)$ has infinite order if $|n| \not\equiv 1 \pmod 8$ is an odd integer.

\begin{proof}[Proof of Theorem~\ref{finite order}]
As in the statement, let $Y = S^3_n(K)$, and suppose that $mY$ bounds a rational homology ball $W$.
Up to changing the orientation of $Y$, we can assume that $n$ is positive;
then, by Example~\ref{n-surgery}, we know that $\rho_Y(\ft_0) = \frac{n-1}4$.
Therefore, $\rho_Y(\ft_0) = 0 \in \Q/2\Z$ if and only if $n \equiv 1 \pmod 8$.

Since $\#^m\ft_0$ is the spin structure $\ft_0$ on $mY$, we know that it extends to $W$, by Lemma~\ref{l:spin-extension}, and hence
\[
m\cdot d(Y,\ft_0) = d(mY,\ft_0) = 0.
\]
From this we get that $\rho_Y(\ft_0) = 0$, which implies $n\equiv 1 \pmod 8$.
\end{proof}

The following is our key computation 
of correction terms.

\begin{proposition}\label{Marco todo}
Let $Y$ be a $\Z/2\Z$--homology $3$--sphere such that the $p$--primary part $H$ of $H_1(Y)$ is $\Z/p^n\Z$, where $n$ is odd and $p$ is a prime congruent to $3$ modulo $4$.
Then the correction terms associated to $ap^{(n-1)/2} \in H$ for $0\leq a \leq (n+1)/2$ are not constant modulo $2$.
\end{proposition}

\begin{proof}
By Lemma~\ref{linking}, the linking form $\lambda_Y$ does not vanish identically on the subgroup of $H$ generated by $x = p^{(n-1)/2} \in H$.
As we saw in Section~\ref{s:metabolizers}, this also means that 
\[
\rho_Y(\ft_x) - \rho_Y(\ft_0) \neq 0.
\]
But this, in turn, implies that $d$ is not constant modulo 2 along $H$.
\end{proof}

We give the following technical statement; it will underpin the proof of Theorem~\ref{linear independence}.

\begin{theorem}\label{main}
Let $Y$ be a $\Z/2\Z$--homology $3$--sphere such that the $p$--primary part of $H_1(Y)$ is $\Z/p^n\Z$, where $n$ is odd and $p$ is a prime congruent to $3$ modulo $4$.
If $N$ is any $\Z/2\Z$--homology $3$--sphere with $H_1(N;\Z/p\Z) = 0$, then $mY \# N$ is nonzero in $\Theta^3_\Q$ for any nonzero $m$.
\end{theorem}


\begin{proof}
Suppose $\overline{Y} = mY\#N$ bounds a rational homology ball for some $Y$ and $N$ as in the statement of the theorem.
Up to reversing the orientation of $\overline{Y}$, we can assume that $m$ is positive.
For convenience, let $H$ denote the $p$--primary part of $H_1(Y)$, viewed as a subset of $H_1(Y)$.

By Theorem~\ref{summaryGL}, $m$ is even, and there is a metabolizer $\overline{M}$ of $(H_1(\overline{Y}),\lambda_{\overline{Y}})$ on which $d$ (and hence $\bar{d}$) vanishes identically.
Let $M$ be the $p$--primary part of $\overline{M}$.
This is a subgroup of the $p$-primary part $G$ of $H_1(\overline{Y})$, which is $H^{\oplus m}$;
the latter is, in turn, isomorphic to $(\Z/p^n\Z)^{\oplus m}$.

In fact, $M$ is a metabolizer for $(G,\lambda)$, where $\lambda$ is the restriction of $\lambda_{\overline{Y}}$ to $G\times G$.
Since $\bar{d}$ vanishes on $\overline{M}$, it vanishes also on $M\subset \overline{M}$

Let $f \co H \to \Q$ be defined by $f(h) = \bar{d}(Y,\ft_h) = d(Y,\ft_h) - d(Y, \ft_0)$.
As in Section~\ref{algebra}, we can then define $f^{(m)} \co H^{\oplus m} \rightarrow \Q$ by
$f^{(m)}(h_1,\dots,h_m)= f(h_1)+\dots +f(h_m)$.
Note that
\begin{flalign*}
f^{(m)}(h_1,\dots,h_m) &= (d(Y,\ft_{h_1})- d(Y,\ft_{0}))+\dots +(d(Y,\ft_{h_m})- d(Y,\ft_{0})) &&\\
& =  (d(Y,\ft_{h_1}) + \dots + d(Y,\ft_{h_m}) + d(N,\ft_0)) - (md(Y,\ft_0) + d(N,\ft_0)) &&\\ 
& = d(mY \# N, \ft) - d(mY \# N, \ft_{0}) = \bar{d}(\overline{Y},\ft),
\end{flalign*}
where $\ft = \ft_{h_1}\#\dots\#\ft_{h_m}\#\ft_0$, and the last connected summand, $\ft_0$, is the spin structure on $N$.

It follows that $f^{(m)}$ is identically zero on $M$ (cf. \cite[Theorem 3.2]{HLR}).
Then $f^{(m)}$ satisfies the hypotheses of Proposition~\ref{metabolizers}, and so we can conclude that $f$ vanishes identically on the subgroup of $H$ of order $p^{(n+1)/2}$.
But since $f(h) = d(Y,\ft_{h}) - d(Y,\ft_{0})$, this contradicts Proposition~\ref{Marco todo}.
\end{proof}

Now we give the proof of Theorem~\ref{linear independence}.

\begin{proof}[Proof of Theorem~\ref{linear independence}]
Consider $\overline{Y} = a_1Y_{1}\#a_2Y_{2}\# \dots \#a_nY_{n}\#Z$ for $Z$ an integral homology sphere and some non-zero integers $a_1,\dots,a_n$ (and some reindexing of the $Y_i$).
The statement is equivalent to proving that $\overline{Y}$ is non-zero in $\Theta^3_\Q$.
Then $Y=Y_{1}$, $m=a_1$, and $N=a_2Y_{2}\# \dots \#a_nY_{n}\#Z$ satisfy the assumptions of Theorem~\ref{main}, and therefore $\overline{Y}$ is non-zero in $\Theta^3_\Q$.
\end{proof}

Theorem~\ref{main} translates into the following corollary about knot concordance, which is slightly more general than Corollary~\ref{concordance}.
Note that $\det K$ denotes the determinant of $K$, and recall that $|H_1(Y_K)| = \left|\det K\right|$.

\begin{corollary}[\cite{LN1, LN2}]\label{main concordance}
Let $K$ be a knot in $S^3$ such that the $p$--primary part of $H_1(Y_K)$ is $\Z/p^n\Z$, where $n$ is odd and $p$ is a prime congruent to $3$ modulo $4$.
If $J$ is any other knot such that $p$ does not divide $\det J$, then $mK \# J$ is nonzero in $\calc$ for any $m$.
\end{corollary}

\begin{proof}
Suppose such a knot $\overline{K} = mK \# J$ is slice.
Then the double cover of $B^4$ branched over the slice disk for $\overline{K}$ is a rational homology ball with boundary $Y_{\overline{K}}= mY_K \# Y_J$, which contradicts Theorem~\ref{main}.
\end{proof}

The above statement when $J$ is the unknot is exactly Theorem 1.2 in~\cite{LN2},
but the full statement of Corollary~\ref{main concordance} follows from their argument as in Theorem 9.4 of \cite{LN1}
(which is the case $n=1$).
Finally, Corollary~\ref{concordance finite order} follows from Theorem~\ref{finite order} in the same manner as the above proof, and Corollary~\ref{concordance} follows from either Theorem~\ref{linear independence} or Corollary~\ref{main concordance}.

\end{section}


\bibliographystyle{amsalpha}
\bibliography{terms.bib}

\end{document}